\documentclass[11pt]{article}

\usepackage[margin=1.1in]{geometry}
\usepackage{amsmath,amssymb,amsthm,mathtools}
\usepackage{enumitem}
\usepackage{xcolor}
\usepackage{hyperref}
\usepackage{microtype}

\hypersetup{colorlinks=true,linkcolor=blue,citecolor=blue,urlcolor=blue}

\newtheorem{theorem}{Theorem}[section]
\newtheorem{lemma}[theorem]{Lemma}
\newtheorem{proposition}[theorem]{Proposition}

\newtheorem{conjecture}[theorem]{Conjecture}
\newtheorem{definition}[theorem]{Definition}
\newtheorem{remark}[theorem]{Remark}
\newtheorem{example}[theorem]{Example}

\newcommand{\COP}{\operatorname{COP}}
\newcommand{\MK}{\mathcal{K}}
\newcommand{\MQ}{\mathcal{Q}}

\newcommand{\dist}{\operatorname{dist}}

\newcommand{\one}{\mathbf{1}}
\newcommand{\rev}[2]{\textcolor{blue}{#1}}

\title{Sum-of-Squares Certificates for Copositive Matrices via Recursive Identities: The de Klerk-Pasechnik Conjecture and Hoffman-Pereira Matrices}
\author{
  Jineon Baek \thanks{
  Korea Institute for Advanced Study \texttt{jineon.kias@gmail.com}}
  \and
  Luis Felipe Vargas \thanks{Université de Toulouse; LAAS-CNRS, 7 avenue du colonel Roche, F-31400 
Toulouse, France
\texttt{lfvargasbe@laas.fr}}
}
\date{\today}

\begin{document}
\maketitle

\begin{abstract}
We establish the conjecture by de Klerk and Pasechnik (2002),  claiming that the semidefinite bounds $\vartheta^{(r)}(G)(r\geq 0)$ for the stability number $\alpha(G)$ are exact at $r=\alpha(G)-1$, by exhibiting an explicit sum-of-squares certificate. This certificate allows us to recover a known characterization of the  minimizers of the Motzkin-Straus formulation for $1/\alpha(G)$. Additionally, we give sum-of-squares copositivity certificates for the matrices satisfying the Hoffman--Pereira sign condition, a crucial condition for characterizing copositive matrices with $\{-1,0,1\}$ entries.
\end{abstract}


\section{Introduction}
\label{sec:introduction}

A real symmetric matrix $M\in\mathcal S^n$ is \emph{copositive} if
\[
y^{\mathsf T} M y \ge 0 \qquad \text{for all } y\in\mathbb R^n_+,
\]
where $\mathbb R^n_+=\{x\in \mathbb{R}^n: x_i\geq 0, \text{ for }i\in [n]\}$ denotes the nonnegative orthant. The set of copositive matrices,
\[
\mathrm{COP}_n = \{M\in\mathcal S^n : y^{\mathsf T}My\ge0 \text{ for all } y\in\mathbb R^n_+\},
\]
is a closed convex cone. Although $\COP_n$ looks innocently similar to the cone $\mathcal{S}_n^+$ of positive semidefinite matrices, it captures many difficult problems.
Indeed, many NP-hard combinatorial optimization problems, including the stability number and the chromatic number of a graph, can be formulated exactly as linear programs over $\mathrm{COP}_n$~\cite{deKlerkPasechnik2002,GvozdenovicLaurent08}. Moreover, testing whether a matrix is copositive is co-NP-complete~\cite{MurtyKabadi1987}. Copositive programming is therefore a hard but remarkably expressive framework, and this motivates the search for explicit \emph{certificates} of copositivity.

Given a symmetric matrix $M\in\mathcal S^n$, we consider the polynomials
\begin{equation}
p_M(x) = x^{\mathsf T}Mx, \qquad P_M(x) = (x^{\circ2})^{\mathsf T}M\,x^{\circ2},
\label{eq:pM-PM}
\end{equation}
where $x^{\circ2}=(x_1^2,\dots,x_n^2)$. Thus $M$ is copositive precisely when $p_M$ is nonnegative on $\mathbb R^n_+$, or, equivalently, when $P_M$ is nonnegative on $\mathbb R^n$. In this paper we study \emph{sum-of-squares} certificates of copositivity for two well-known classes of copositive matrices: 1) Copositive matrices arising from graphs and their stability number and 2) Hoffman--Pereira type matrices.

A polynomial $p$ is a sum of squares if $p=\sum_i q_i^2$ for some polynomials $q_i$. We denote the cone of sum-of-squares polynomials by $\Sigma$. We set $\Sigma_r=\Sigma\cap \mathbb{R}[x_1, \dots, x_n]_r$.

Observe that if, for some $r\geq0$,
\begin{align}\label{cert-parrilo}
\Big(\sum_{i=1}^n x_i^2\Big)^rP_M(x)\in \Sigma,
\end{align}
then $M$ is copositive. This certificate was first proposed by Parrilo (for $r=1$)~\cite{parrilo}. Later, de Klerk and Pasechnik \cite{deKlerkPasechnik2002} defined, for $r\geq0$, the cones $\MK_n^{(r)}$ based on this certificate for optimization purposes.
\begin{align}
  \mathcal K_n^{(r)} = \Big\{\, M\in\mathcal S^n \;:\; \Big(\sum_{i=1}^n x_i^2\Big)^{r} (x^{\circ2})^{\mathsf T} M\, x^{\circ2} \in \Sigma \,\Big\}.  
\end{align}
It is clear that $\MK_n^{(r)}\subseteq\MK_n^{(r+1)}\subseteq\COP_n$ for all $r\geq 0$, and
these cones are known to cover the interior of $\COP_n$, i.e.,
$\mathrm{int}(\COP_n)\subseteq \bigcup_{r\geq0}\MK_n^{(r)}$. This inclusion is a consequence
of a classical theorem of P\'olya~\cite{Polya}: if $q$ is a homogeneous polynomial that is
strictly positive on $\mathbb R^n_+\setminus\{0\}$, then, for a sufficiently large integer
 $r$, the polynomial $\big(\sum_{i=1}^n x_i\big)^r q(x)$ has only nonnegative
coefficients. Applying this to $q=p_M$ for $M$ in the interior of $\COP_n$ (so that $p_M$ is
strictly positive on $\mathbb R^n_+\setminus\{0\}$) shows that, for $r$ large enough,
$\big(\sum_i x_i\big)^r\,p_M(x)$ has only nonnegative coefficients, and hence
$\big(\sum_i x_i^2\big)^r\,P_M(x)$ is trivially a sum of squares of monomials, so
$M\in\MK_n^{(r)}$.

Diananda \cite{Dia} showed that, for $n\leq 4$, every $n \times n$ copositive matrix $M$ can
be written as $M = P + N$, where $P$ is positive semidefinite and $N$ is entrywise nonnegative. Choi and Lam~\cite{CL} showed that, for every $n\geq1$, the cone $\MK_n^{(0)}$
consists precisely of the matrices that admit such a decomposition~\cite{HallNewman1963}. Therefore, for $n \leq 4$, we have $\COP_n=\MK_n^{(0)}$. This result does not extend to $n=5$~\rev{\cite{HallNewman1963}}{\cite{}}. For example, the Horn matrix is a $5\times 5$ copositive matrix for which $P_M$ is not a sum of squares. Schweighofer and Vargas~\cite{SV24} showed, however, that every copositive $5\times5$ matrix admits a certificate as in~\eqref{cert-parrilo}; that is, $\COP_5=\bigcup_{r\geq 0}\MK_5^{(r)}$. Dickinson, D\"ur, Gijben, and Hildebrand
\cite{DDGH} showed that, for every fixed $r\geq0$, some positive diagonal scaling of the
Horn matrix, while copositive, fails to lie in $\MK_5^{(r)}$, so the union above is necessarily
infinite. The same fact also follows from the stronger result of Bodirsky, Kummer, and Thom~\cite{BodirskyKummerThom2026}, which shows that $\COP_5$ is not a spectrahedral shadow.

Several other certificates, and the corresponding cones, have been proposed in the
literature. A related
hierarchy, working directly with $p_M$ rather than $P_M$, was
introduced by Pe\~na, Vera, and Zuluaga \cite{ZVP06}:
\begin{align}\label{def-Q}
\mathcal Q_n^{(r)} = \Big\{\, M\in\mathcal S^n \;:\; \Big(\sum_{i=1}^n x_i\Big)^{r} x^{\mathsf T}Mx
= \sum_{\substack{\beta\in\mathbb N^n \\ |\beta|\in\{r,\,r+2\}}} \sigma_\beta\, x^\beta,
\ \ \sigma_\beta\in\Sigma_{r+2-|\beta|} \,\Big\}.  
\end{align}
As with $\MK_n^{(r)}$, we have $\MQ_n^{(r)}\subseteq\MQ_n^{(r+1)}\subseteq\COP_n$ for all
$r\geq0$, and $\mathrm{int}(\COP_n)\subseteq\bigcup_{r\geq0}\MQ_n^{(r)}$. Moreover, $\MQ_n^{(r)}\subseteq \MK_n^{(r)}$ for all $r\geq0$: if $M\in\MQ_n^{(r)}$, substituting $x_i\mapsto x_i^2$ in the decomposition certifying
$M\in\MQ_n^{(r)}$ yields a sum-of-squares decomposition of
$\big(\sum_i x_i^2\big)^r P_M(x)$, showing $M\in\MK_n^{(r)}$.

Vargas, Vera, and Dickinson~\cite{VargasVeraDickinson2025} introduced a further
hierarchy of inner approximations to $\COP_n$, obtained by relaxing the fixed multiplier
$\big(\sum_{i=1}^n x_i\big)^r$ in the definition of $\MQ_n^{(r)}$ to an arbitrary homogeneous
polynomial with nonnegative coefficients. For $r\in\mathbb N$, let
$\mathcal N_{n,r} = \{\, \sum_{\beta \in \mathbb N^n,\, |\beta| = r} c_\beta x^\beta \;:\; c_\beta \ge 0 \,\}$
denote the cone of homogeneous degree-$r$ polynomials in $n$ variables with nonnegative
coefficients, and set $\|p\|_1 := \sum_\beta |c_\beta|$ for $p=\sum_\beta c_\beta x^\beta$.
The cones $\widetilde{\MQ}_n^{(r)}$ are then defined by
\begin{equation}
\begin{aligned}
\widetilde{\MQ}_n^{(r)} = \Big\{\, M \in \mathcal S^n \;:\;&
p(x)\, x^{\mathsf T}Mx = \sum_{\substack{\beta \in \mathbb N^n \\ |\beta| \in \{r,\,r+2\}}}
\sigma_\beta\, x^\beta,\quad \sigma_\beta \in \Sigma_{r+2-|\beta|},\\
&\text{for some } p \in \mathcal N_{n,r},\quad \|p\|_1 = 1 \,\Big\}.
\end{aligned}
\label{eq:tildeQ}
\end{equation}
In words, $M\in\widetilde{\MQ}_n^{(r)}$ if \emph{some} normalized nonnegative-coefficient
multiplier $p$ of degree $r$ (not necessarily $\big(\sum_i x_i\big)^r$) makes
$p(x)\,x^{\mathsf T}Mx$ admit a structured decomposition of the same shape as in the
definition of $\MQ_n^{(r)}$. Since $\big(\sum_i x_i\big)^r\in\mathcal N_{n,r}$ up to
normalization, it follows that
\[
\mathcal Q_n^{(r)} \;\subseteq\; \widetilde{\mathcal Q}_n^{(r)} \;\subseteq\; \COP_n.
\]

\subsection{The stability number and the de Klerk--Pasechnik conjecture}

Let $G=(V=[n],E)$ be a graph on $n$ vertices. A set $S\subseteq V$ is \emph{stable} (or independent) if it contains no edge of $G$, and the \emph{stability number} $\alpha(G)$ is the largest cardinality of a stable set. Computing $\alpha(G)$ is NP-hard. De Klerk and Pasechnik~\cite{deKlerkPasechnik2002} proposed the following exact reformulation of $\alpha(G)$ as a linear optimization problem over $\COP_n$. Let $A_G$ be the adjacency matrix of $G$, $I$ the identity matrix, and $J$ the all-ones matrix. Then
\begin{align}\label{alpha-COP}
\alpha(G) = \min\big\{\, t \;:\; t(I+A_G) - J \in \mathrm{COP}_n \,\big\}.
\end{align}
For $t=\alpha(G)$, the minimizing matrix
\[
M_G = \alpha(G)(I+A_G) - J
\]
is therefore copositive. De Klerk and Pasechnik introduced the hierarchy $\vartheta^{(r)}(G)$ (for $r\geq 0$) by replacing $\COP_n$ with $\MK_n^{(r)}$ in formulation (\ref{alpha-COP}).
\[
\vartheta^{(r)}(G) = \min\big\{\, t \;:\; t(I+A_G) - J \in \mathcal K_n^{(r)} \,\big\}.
\]
Clearly,
\[
\alpha(G)\leq \dots \leq \vartheta^{(2)}(G) \leq \vartheta^{(1)}(G) \leq \vartheta^{(0)}(G).
\]
Since $\operatorname{int}(\COP_n)\subseteq\bigcup_{r\geq0}\MK_n^{(r)}$, it follows that $\lim_{r\to \infty}\vartheta^{(r)}(G)=\alpha(G)$. The parameter $\vartheta^{(0)}(G)$ coincides with the parameter $\vartheta'(G)$ introduced by Schrijver~\rev{\cite{Schrijver1979}}{\cite{}} as a strengthening of the Lov\'asz $\vartheta$ number~\cite{Lovasz}. De Klerk and Pasechnik~\cite{deKlerkPasechnik2002} conjectured that this hierarchy converges to $\alpha(G)$ after $\alpha(G)-1$ steps or, in other words, that $M_G$ admits a certificate as in~\eqref{cert-parrilo} with $r=\alpha(G)-1$. Whether $M_G$ admits such a certificate is not immediate from the results of P\'olya~\cite{Polya} or Reznick~\cite{reznick}, or from the fact that the cones $\MK^{(r)}$ cover the interior of $\COP_n$. Indeed, $M_G$ lies on the \emph{boundary} of $\COP_n$, so none of these results applies directly.
\begin{conjecture}[de Klerk--Pasechnik, 2002]
\label{conj:dkp}
For every graph $G$,
\[
\vartheta^{(\alpha(G)-1)}(G) = \alpha(G).
\]
Equivalently, the graph matrix $M_G=\alpha(G)(I+A_G)-J$ lies in $\mathcal K_n^{(\alpha(G)-1)}$.
\end{conjecture}
The hierarchy $\vartheta^{(r)}(G)$ and Conjecture~\ref{conj:dkp} have been extensively studied. We recap some of the main results. Conjecture~\ref{conj:dkp} is known to hold for perfect graphs (with $r=0$) and for odd cycles (with $r=1$). Gvozdenovi\'c and Laurent~\cite{GL07} proved Conjecture~\ref{conj:dkp} for graphs with $\alpha(G)\leq 8$ (see also~\cite{PVZ07} for $\alpha(G)\leq 6$). Schweighofer and Vargas~\cite{SV24} proved that $\vartheta^{(r)}(G)$ always has finite convergence; that is, the matrices $M_G$ always admit a certificate as in~\eqref{cert-parrilo}. However, their result gives no bound on the number of steps required.

In this paper, we prove the de Klerk--Pasechnik conjecture (Conjecture~\ref{conj:dkp}) by exhibiting an explicit sum-of-squares decomposition. This decomposition in fact shows the stronger statement that $M_G$ lies in the cone $\mathcal Q_n^{(\alpha(G)-1)}$ for \emph{every} graph $G$. 
\begin{theorem}
For every graph $G=([n],E)$, the matrix $M_G$ belongs to $\MQ_n^{(\alpha(G)-1)}$. In particular, Conjecture~\ref{conj:dkp} holds.
\end{theorem}

We observe that the degree in this result is tight. Indeed, Vargas \cite{Vargas-thesis} (see also \cite{VargasVeraDickinson2025}) constructed a class of graphs $L_k$ (for $k\geq2$), with $\alpha(L_k)=k$, such that $M_{L_k}\notin\mathcal Q^{(k-2)}$.

\paragraph{Minimizers of the Motzkin-Straus formulation} Formulation (\ref{alpha-COP}) can be seen also as a reformulation of a classical result of Motzkin and Straus \cite{motzkin-straus}, who showed that
\[
\frac{1}{\alpha(G)} = \min\Big\{\, x^{\mathsf T}(I+A_G)x \;:\; x\in\Delta_n \,\Big\},
\]
where $\Delta_n=\{x\in\mathbb R^n_+ : \sum_i x_i=1\}$. The zeros of $x^{\mathsf T}M_Gx$ on $\Delta_n$ are exactly the minimizers of this program. These zeros form an obstruction to obtaining a certificate of membership of $M_G$ in $\mathcal Q_n^{(r)}$: if $z\in\Delta_n$ satisfies $p_G(z)=0$, then the right-hand side of the identity certifying that $M_G\in \MQ_n^{(r)}$ (see~\eqref{def-Q}) must also vanish at $z$. Since that right-hand side is a sum of manifestly nonnegative terms (weighted squares and nonnegative-coefficient polynomials), its vanishing forces \emph{every} individual term to vanish. This is a highly restrictive condition, satisfied only by very particular configurations of zeros. As we show in Section~\ref{sec:minimizers}, our identity (Theorem~\rev{\ref{thm:main-identity}}{\ref{main-identity}}) in fact allows us to recover a full characterization of these zeros directly, recovering a known structural description of the Motzkin--Straus minimizers~\cite{LV21a}. This same obstruction is exploited in~\cite{Vargas-thesis,VargasVeraDickinson2025} to construct graphs $L_k$ for which the degree needed for convergence is large, showing that the exponent $\alpha(G)-1$ in Theorem~\ref{thm:main-identity} cannot be improved in general.

\subsection{Hoffman--Pereira copositive matrices}

The second class of copositive matrices studied in this paper goes back to the work of Hoffman and Pereira~\cite{HoffmanPereira1973}. They characterized copositivity for symmetric matrices with unit diagonal and off-diagonal entries in $\{-1,0,1\}$. For such a matrix $A\in \mathcal{S}^n$, let $G_-(A)$ denote its \emph{negative-entry graph}, the graph on $[n]$ whose edges are the pairs $\{i,j\}$ with $a_{ij}=-1$. Hoffman and Pereira proved that $A$ is copositive if and only if $G_-(A)$ is triangle-free and $a_{ij}=1$ whenever $i$ and $j$ are at distance two in $G_-(A)$.

We refer to these requirements as the \emph{Hoffman--Pereira sign condition}. The smallest nontrivial instance is already a distinguished object: taking $G_-(A)$ to be the $5$-cycle $C_5$ yields the Horn matrix~\cite{Dia,HallNewman1963}, historically the first example of a copositive matrix that is not the sum of a positive semidefinite matrix and an entrywise nonnegative one or, equivalently, does not belong to $\MK_n^{(0)}$. These matrices are also relevant to the geometry of the copositive cone: Hoffman and Pereira characterized the extremal matrices in this class, and Hildebrand later showed that their diagonal scalings describe the extremal copositive matrices whose minimal zeros all have support of cardinality two~\cite{HoffmanPereira1973,Hildebrand2018}.

Our second contribution is a certificate-driven proof of the sufficiency part of the Hoffman--Pereira theorem. We prove that every matrix $A\in \mathcal{S}^n$ satisfying the Hoffman--Pereira sign condition belongs to $\widetilde{\MQ}_n^{(r)}$ for some $r\geq0$. Therefore, these matrices are not only copositive but also admit a sum-of-squares certificate of copositivity. The starting point is a local identity (Theorem~\ref{thm:local-identity}): for every principal submatrix $A[U]$, the product $S_U(x)\,p_U(x)$ of the quadratic form $p_U(x)=x_U^{\mathsf T}A[U]x_U$ with the linear form $S_U(x)=\sum_{i\in U}x_i$ decomposes into a sum of monomial-times-square terms, the analogous forms $p_{Z_U(i)}$ associated with the ``zero-neighborhoods'' $Z_U(i)=\{j\in U\setminus\{i\}: a_{ij}=0\}$, and a residual cubic that, under the sign condition, has only nonnegative coefficients. Iterating this identity yields, by induction, a sum-of-squares certificate for $A$.

We then analyze whether the cones $\MK^{(r)}$ and $\MQ^{(r)}$ certify matrices satisfying the Hoffman--Pereira sign condition. For every $n\geq7$, we construct an $n\times n$ matrix satisfying the sign condition that does \emph{not} belong to any cone $\MK^{(r)}$ (and hence to no $\MQ^{(r)}$). On the positive side, we prove that a matrix with at most four zeros per row belongs to $\MQ_n^{(1)}$. Consequently, the classical $7\times7$ Hoffman--Pereira matrix admits such a certificate.

\paragraph{Disclosure}
The identities of Theorem \ref{thm:main-identity} and Theorem \ref{thm:local-identity}, their proof strategies and most of the original proof text were produced by OpenAI's GPT-5.6 through ChatGPT in response to prompts from Jineon Baek. The authors checked the proofs, adapted the terminology, and substantially rewrote parts of the proofs. The authors take full responsibility for the content of this paper.
\section{Graph matrices \texorpdfstring{$M_G$}{M G} and the de Klerk--Pasechnik conjecture}
\label{sec:main-identity}

\subsection{Notation and main theorem}
\label{sec:notation}

We first introduce the notation used throughout the construction. Let $G=(V=[n],E)$ be a finite simple graph on vertex set $V=[n]$, and let
\[
\alpha := \alpha(G)
\]
be its stability number. Recall the graph matrix
\[
M_G = \alpha(I+A_G) - J
\]
from the Introduction, where $A_G$ is the adjacency matrix of $G$, $I$ is the identity matrix, and $J$ is the all-ones matrix. Its associated quadratic form is
\[
p_G(x) = x^{\mathsf T} M_G x
= \alpha\sum_{i\in V} x_i^2 + 2\alpha\!\!\sum_{\{i,j\}\in E}\!\! x_ix_j - S^2,
\qquad S := \sum_{i\in V} x_i.
\]

For $r\ge1$, let $\mathcal S_r(G)$ denote the set of \emph{stable sets} (independent sets) of $G$ of size $r$, i.e., subsets $T\subseteq V$ with $|T|=r$ such that no two vertices of $T$ are adjacent in $G$. When $G$ is clear from context we write simply $\mathcal S_r$. For a stable set $T$, write
\[
x_T = \prod_{i\in T} x_i.
\]

For a stable set $T$, define
\[
\nu_T(x) = \sum_{\substack{v\notin T \\ T\cup\{v\}\,\in\,\mathcal S_{|T|+1}}} x_v,
\quad \quad \text{ and} \quad \quad 
\rho_T(x) = \sum_{\substack{v\notin T \\ v \text{ adjacent to at }\\\text{least two vertices of } T}} x_v.
\]

For a stable set $T$ with $|T|=r$, define the \emph{residual}
\[
L_T(x) = (\alpha-r)\,S - \alpha\,\nu_T(x).
\]

Finally, define constants $\lambda_1=1$ and, recursively,
\[
\lambda_{r+1} = \lambda_r\,\frac{\alpha r}{\alpha-r},
\qquad\text{equivalently}\qquad
\lambda_r = \prod_{t=1}^{r-1}\frac{\alpha t}{\alpha-t}.
\]

We now prove an algebraic identity that shows Conjecture \ref{conj:dkp}.

\begin{theorem}
\label{thm:main-identity}
For every graph $G$ with $\alpha:=\alpha(G)\ge1$,
\[
S^{\alpha-1} p_G(x) = \sum_{r=1}^{\alpha-1} \lambda_r S^{\alpha-1-r}
\left[\frac{1}{\alpha-r}\sum_{T\in\mathcal S_r} x_T L_T(x)^2
+ \frac{\alpha^2}{\alpha-r}\sum_{D\in\mathcal S_{r+1}} x_D\,\rho_D(x)\right].
\]
In particular, $M_G\in \mathcal{Q}_n^{(\alpha-1)}$.
\end{theorem}

Note that if $\alpha(G)=1$, then $G$ is complete so $p_G(x) = 0$ and $M_G=0\in\mathcal Q_n^{(0)}$.

\subsection{Proof of Theorem \ref{thm:main-identity}}

We begin with a combinatorial identity that expresses the residual of a stable set in terms of the residuals of its one-element deletions.

\begin{lemma}
\label{lem:combinatorial}
For every independent set $D$ of $G$ with $|D|=r+1$,
\[
\sum_{u\in D} L_{D\setminus\{u\}}(x) = r\,L_D(x) + \alpha\,\rho_D(x).
\]
\end{lemma}

\begin{proof}
Recall that  $L_{D\setminus\{u\}}(x)= (\alpha-r)S-\alpha\nu_{D\setminus\{u\}}(x)$. We first compute $\sum_{u\in D}\nu_{D\setminus\{u\}}(x)$. Since this is linear in $x$, it suffices to compute
the coefficient of $x_v$ for each vertex $v$, namely the number of sets
$D\setminus\{u\}$ ($u\in D$) that $v$ extends to an independent set.
\begin{itemize}
\item \textbf{$v\in D$.} Removing $u=v$ from $D$ is the only removal that can place
$v$ back into the resulting set; indeed $D\setminus\{u\}\cup\{v\}$ equals $D$ itself
precisely when $u=v$, which is independent, while for $u\neq v$ the vertex $v$ is
already present in $D\setminus\{u\}$. So $v$ extends exactly one of the $r+1$ sets.
Coefficient $1$.

\item $v\notin D$ and $v$ is not adjacent to any vertex of $D$. Then removing
any single vertex $u$ from $D$ still leaves $v$ nonadjacent to all of
$D\setminus\{u\}$, so $v$ extends every one of the $r+1$ sets. Coefficient $r+1$.
These are exactly the vertices counted in $\nu_D(x)$.

\item $v\notin D$ and $v$ is adjacent to exactly one vertex of $D$. Removing
that one vertex is the only way to make $v$ nonadjacent to all of $D\setminus\{u\}$,
so $v$ extends exactly one such set. Coefficient $1$.

\item $v\notin D$ and $v$ is adjacent to at least two vertices of $D$. No
single removal can eliminate every adjacency, so $v$ extends none of the
$D\setminus\{u\}$. Coefficient $0$. These are exactly the vertices counted in
$\rho_D(x)$.
\end{itemize}
Every vertex falls into exactly one case, and $S$ counts every vertex once, so
\begin{equation}
\sum_{u\in D}\nu_{D\setminus\{u\}}(x) = S + r\,\nu_D(x) - \rho_D(x).
\label{eq:ddagger}
\end{equation}
Since $|D\setminus\{u\}|=r$ for each $u\in D$, we have
$L_{D\setminus\{u\}}(x) = (\alpha-r)S - \alpha\,\nu_{D\setminus\{u\}}(x)$. Summing this over the
$r+1$ elements $u\in D$, the term $(\alpha-r)S$ does not depend on $u$ and is simply
repeated $r+1$ times, giving
\[
\sum_{u\in D} L_{D\setminus\{u\}}(x)
= (r+1)(\alpha-r)S - \alpha\sum_{u\in D}\nu_{D\setminus\{u\}}(x).
\]
Substituting \eqref{eq:ddagger} for the remaining sum,
\begin{align*}
\sum_{u\in D} L_{D\setminus\{u\}}(x)
= &(r+1)(\alpha-r)S - \alpha\big(S + r\,\nu_D(x) - \rho_D(x)\big) \\
= &\big[(r+1)(\alpha-r)-\alpha\big]S - \alpha r\,\nu_D(x) + \alpha\rho_D(x).
\end{align*}
Finally, expanding $(r+1)(\alpha-r)-\alpha = \alpha r-r^2+\alpha-r-\alpha = r(\alpha-r-1)$ gives
\[
\sum_{u\in D} L_{D\setminus\{u\}}(x)
= r(\alpha-r-1)S - \alpha r\,\nu_D(x) + \alpha\rho_D(x).
\]
Since $|D|=r+1$, $L_D(x) = (\alpha-r-1)S - \alpha\nu_D(x)$, so the right-hand side equals
$r\,L_D(x) + \alpha\,\rho_D(x)$.
\end{proof}

With Lemma~\ref{lem:combinatorial} in hand, we can now aggregate the residual $L_T$
over all independent sets of a fixed size and relate consecutive sizes. For
$1\le r\le \alpha$, define
\begin{align}
\label{def:A_r}
A_r(x) \;=\; \sum_{T\in\mathcal S_r(G)} x_T\, L_T(x).
\end{align}

We observe that the base case $A_1$ recovers the original quadratic form $p_G(x)$ and that $A_\alpha(x)=0$.
\begin{lemma}
\label{lem:A1}
We have that $A_1(x) = p_G(x)$, and $A_\alpha(x)=0$.
\end{lemma}

\begin{proof}

We first show that $A_1(x)=p_G(x)$. For $r=1$, each $T=\{i\}$ is a single vertex, and $\nu_{\{i\}}(x)=\sum_{j\neq i:\,\{i,j\}\notin E(G)}x_j$,
so
\[
L_{\{i\}}(x) = (\alpha-1)S - \alpha\sum_{j\not\sim i} x_j.
\]
Summing $x_i L_{\{i\}}(x)$ over all $i$ and using that each non-edge $\{i,j\}\notin
E(G)$ is counted twice in $\sum_i x_i\sum_{j\not\sim i}x_j$,
\[
A_1(x) = (\alpha-1)S^2 - 2\alpha\sum_{\{i,j\}\notin E(G)} x_ix_j.
\]
On the other hand, expanding $p_G(x)=x^{\mathsf T}M_Gx$ from $M_G=\alpha(I+A_G)-J$ and writing
$S^2=\sum_i x_i^2+2\sum_{\{i,j\}\in E(G)}x_ix_j+2\sum_{\{i,j\}\notin E(G)}x_ix_j$, both
expressions reduce to
\[
(\alpha-1)\sum_i x_i^2 + 2(\alpha-1)\sum_{\{i,j\}\in E(G)} x_ix_j - 2\sum_{\{i,j\}\notin E(G)} x_ix_j,
\]
so $A_1(x)=p_G(x)$.

We now show $A_\alpha(x)=0$. Let $T\in\mathcal S_\alpha(G)$. Since $\alpha=\alpha(G)$, no vertex $v\notin T$ can be added to $T$ while
keeping it stable, so $\nu_T(x)=0$ and hence
\[
L_T(x) = (\alpha-\alpha)S - \alpha\,\nu_T(x) = 0.
\]
As this holds for every $T\in\mathcal S_\alpha(G)$, we get $A_\alpha(x)=\sum_{T\in\mathcal S_\alpha(G)} x_T\,L_T(x)=0$.
\end{proof}
The following lemma gives the one-step recursion relating $A_r$ and $A_{r+1}$.
\begin{lemma}
\label{lem:one-step}
For every $1\le r\le \alpha-1$,
\[
S\,A_r \;=\; \frac{1}{\alpha-r}\sum_{T\in\mathcal S_r(G)} x_T\, L_T(x)^2
\;+\; \frac{\alpha r}{\alpha-r}\, A_{r+1}
\;+\; \frac{\alpha^2}{\alpha-r}\sum_{D\in\mathcal S_{r+1}(G)} x_D\, \rho_D(x).
\]
\end{lemma}

\begin{proof}
Fix an independent set $T$ with $|T|=r$. By definition,
\[
L_T(x) = (\alpha-r)S - \alpha\,\nu_T(x),
\]
and since $r<\alpha$ this rearranges to
\[
S = \frac{1}{\alpha-r}L_T(x) + \frac{\alpha}{\alpha-r}\nu_T(x).
\]
Multiplying by $x_T L_T(x)$ and summing over $T\in\mathcal S_r(G)$ gives
\begin{equation}
S A_r = \frac{1}{\alpha-r}\sum_{T\in\mathcal S_r(G)} x_T L_T(x)^2
+ \frac{\alpha}{\alpha-r}\sum_{T\in\mathcal S_r(G)} x_T\,\nu_T(x)\,L_T(x).
\label{eq:star}
\end{equation}
By definition of $\nu_T$, we have
\[
x_T\,\nu_T(x) = \sum_{\substack{v\notin T \\ T\cup\{v\}\in\mathcal S_{r+1}(G)}} x_T x_v.
\]
Then, we have that

\[\sum_{T\in\mathcal S_r(G)} x_T\,\nu_T(x)\,L_T(x)
= \sum_{T\in\mathcal S_r(G)}\ \sum_{\substack{v\notin T \\ T\cup\{v\}\in\mathcal S_{r+1}(G)}} x_T x_v\, L_T(x)\]
Every pair $(T,v)$ appearing in the last sum (with $T\in\mathcal S_r(G)$, $v\notin T$,
and $T\cup\{v\}\in\mathcal S_{r+1}(G)$) determines a set
$D:=T\cup\{v\}\in\mathcal S_{r+1}(G)$ together with a distinguished element
$u:=v\in D$, and $T=D\setminus\{u\}$. Conversely, any $D\in\mathcal S_{r+1}(G)$
together with any choice of $u\in D$ gives back a valid pair
$(T,v)=(D\setminus\{u\},\,u)$: indeed $T=D\setminus\{u\}$ has size $r$, $v=u\notin T$,
and $T\cup\{v\}=D$ is independent of size $r+1$. These two maps are inverse to each
other, so summing over pairs $(T,v)$ is the same as summing over pairs $(D,u)$ with
$D\in\mathcal S_{r+1}(G)$, $u\in D$:
\[
\sum_{T\in\mathcal S_r(G)}\ \sum_{\substack{v\notin T \\ T\cup\{v\}\in\mathcal S_{r+1}(G)}} x_T x_v\, L_T(x)
= \sum_{D\in\mathcal S_{r+1}(G)}\ \sum_{u\in D} x_{D\setminus\{u\}}\,x_u\,L_{D\setminus\{u\}}(x).
\]
Since $D\setminus\{u\}$ and $\{u\}$ are disjoint and their union is $D$, we have
$x_{D\setminus\{u\}}\,x_u = x_D$, so
\[
\sum_{D\in\mathcal S_{r+1}(G)}\ \sum_{u\in D} x_{D\setminus\{u\}}\,x_u\,L_{D\setminus\{u\}}(x)
= \sum_{D\in\mathcal S_{r+1}(G)} x_D \sum_{u\in D} L_{D\setminus\{u\}}(x).
\]
By Lemma~\ref{lem:combinatorial}, the inner sum equals $r\,L_D(x)+\alpha\,\rho_D(x)$, so
\begin{align*}
\sum_{D\in\mathcal S_{r+1}(G)} x_D \sum_{u\in D} L_{D\setminus\{u\}}(x)
&= \sum_{D\in\mathcal S_{r+1}(G)} x_D\big(r\,L_D(x) + \alpha\,\rho_D(x)\big)
\\&= r\sum_{D\in\mathcal S_{r+1}(G)} x_D\,L_D(x) + \alpha\sum_{D\in\mathcal S_{r+1}(G)} x_D\,\rho_D(x).
\end{align*}
By definition of $A_{r+1}$, the first sum on the right is $A_{r+1}(x)$, so
\[
\sum_{T\in\mathcal S_r(G)} x_T\,\nu_T(x)\,L_T(x) = r\,A_{r+1}(x) + \alpha\sum_{D\in\mathcal S_{r+1}(G)} x_D\,\rho_D(x).
\]
Substituting into \eqref{eq:star} gives the claimed identity.
\end{proof}

Finally, we iterate the identity of Lemma~\ref{lem:one-step} to relate $A_1$
directly to $A_{m+1}$, for any intermediate level $m$. It is convenient to first
package the two nonnegative terms appearing in the one-step identity into a single
quantity. For $1\le r\le \alpha-1$, define
\begin{align}\label{def:B_r}
B_r(x) = \frac{1}{\alpha-r}\sum_{T\in\mathcal S_r(G)} x_T\, L_T(x)^2
\;+\; \frac{\alpha^2}{\alpha-r}\sum_{D\in\mathcal S_{r+1}(G)} x_D\,\rho_D(x).
\end{align}

With this notation, the one-step identity of Lemma~\ref{lem:one-step} reads, for
$1\le r\le \alpha-1$,
\[
S A_r = B_r + \frac{\alpha r}{\alpha-r}\, A_{r+1}.
\]
Recall also the constants $\lambda_1=1$ and $\lambda_{r+1}=\lambda_r\,\frac{\alpha r}{\alpha-r}$,
so that $\frac{\alpha r}{\alpha-r} = \frac{\lambda_{r+1}}{\lambda_r}$, and the one-step identity
can be rewritten as
\begin{equation}
S A_r = B_r + \frac{\lambda_{r+1}}{\lambda_r}\, A_{r+1}, \qquad 1\le r\le \alpha-1.
\label{eq:one-step-lambda}
\end{equation}

We now iterate \eqref{eq:one-step-lambda} starting from $A_1$.

\begin{proposition}
\label{prop:iterated}
For every $1\le m\le \alpha-1$,
\[
S^m A_1(x) = \sum_{r=1}^{m} \lambda_r\, S^{m-r}\, B_r(x) \;+\; \lambda_{m+1}\, A_{m+1}(x).
\]
\end{proposition}

\begin{proof}
We argue by induction on $m$.

For $m=1$, since $\lambda_1=1$, the identity \eqref{eq:one-step-lambda} at
$r=1$ reads
\[
S A_1 = B_1 + \frac{\lambda_2}{\lambda_1} A_2 = \lambda_1 S^0 B_1 + \lambda_2 A_2,
\]
which is exactly the claimed identity for $m=1$.

Suppose the identity holds for some $m$ with
$1\le m\le \alpha-2$, so that
\[
S^m A_1(x) = \sum_{r=1}^{m} \lambda_r S^{m-r} B_r(x) + \lambda_{m+1} A_{m+1}(x).
\]
Multiplying both sides by $S$,
\[
S^{m+1} A_1(x) = \sum_{r=1}^{m} \lambda_r S^{m+1-r} B_r(x) + \lambda_{m+1}\, S A_{m+1}(x).
\]
Since $m+1\le \alpha-1$, identity \eqref{eq:one-step-lambda} applies at $r=m+1$:
\[
S A_{m+1} = B_{m+1} + \frac{\lambda_{m+2}}{\lambda_{m+1}} A_{m+2}.
\]
Multiplying by $\lambda_{m+1}$,
\[
\lambda_{m+1} S A_{m+1}(x) = \lambda_{m+1} B_{m+1}(x) + \lambda_{m+2} A_{m+2}(x).
\]
Substituting this into the previous gives
\begin{align*}
S^{m+1} A_1(x) &= \sum_{r=1}^{m} \lambda_r S^{m+1-r} B_r(x) + \lambda_{m+1} B_{m+1}(x) + \lambda_{m+2} A_{m+2}(x) \\
&= \sum_{r=1}^{m+1} \lambda_r S^{m+1-r} B_r(x) + \lambda_{m+2} A_{m+2}(x),
\end{align*}
which is the claimed identity for $m+1$.
\end{proof}
As a direct consequence we obtain the proof of our main result Theorem \ref{thm:main-identity}.
\begin{proof}[Proof of Theorem~\ref{thm:main-identity}]
Taking $m=\alpha-1$ in Proposition~\ref{prop:iterated} gives
\[
S^{\alpha-1} A_1(x) = \sum_{r=1}^{\alpha-1} \lambda_r S^{\alpha-1-r} B_r(x) + \lambda_\alpha A_\alpha(x).
\]
By Lemma~\ref{lem:A1}, $A_1=p_G$ and $A_\alpha=0$, so the last term vanishes and
\[
S^{\alpha-1} p_G(x) = \sum_{r=1}^{\alpha-1} \lambda_r S^{\alpha-1-r} B_r(x).
\]
Expanding $B_r$ according to \eqref{def:B_r} gives exactly the identity of Theorem~\ref{thm:main-identity}.
To verify membership in $\mathcal Q_n^{(\alpha-1)}$, expand each factor $S^{\alpha-1-r}$ in the displayed identity. Its coefficients are nonnegative. The first family of terms is therefore a sum of monomials of degree $\alpha-1$ times squares of linear forms, while the second family is a homogeneous polynomial of degree $\alpha+1$ with nonnegative coefficients. This is precisely a decomposition of the form required in~\eqref{def-Q}.
\end{proof}

\begin{remark}
In \cite{GL07} (resp. \cite{PVZ07}), the authors prove that the following inequality holds for $1 \leq r \leq \min(\alpha(G)-1, 6)$ and for $r = \alpha(G) - 1 = 7$ (resp. for $r \leq \min(\alpha(G)-1, 5)$):
\begin{equation}\label{ineq-GL}
\vartheta^{(r)}(G)
\leq
r+
\max_{ S\in \mathcal{S}_r}
\vartheta^{(0)}(G\setminus S^\perp),    
\end{equation}
where $S^\perp = \{i \in V : i\in S \text{ or }\{i, j\}\in E \text{ for some }j \in  S\}$ is the extended neighborhood of $S$. This shows Conjecture \ref{conj:dkp} for $\alpha(G) \leq 8$ (resp. $\alpha(G) \leq 6$). Indeed, for $r=\alpha(G)-1$, we have that, for $S\in \mathcal{S}_r$, the graph $G\setminus S^\perp$ is either the empty graph or a clique, and thus $\vartheta^{(0)}(G\setminus S^\perp)\leq 1$. 

The proofs of (\ref{ineq-GL}) in \cite{GL07} and \cite{PVZ07}, for the values of $r$ specified above, rely on a recursive construction of sum-of-squares certificates, in which a certain parameter must satisfy a polynomial inequality. This constraint prevents it from showing the conjecture for graphs with $\alpha(G) \geq 9$.

It is worth noting that inequality \eqref{ineq-GL} provides information beyond the terminal level of the hierarchy, and can in fact give useful bounds on the intermediate relaxations $\vartheta^{(r)}(G)$ as well. Our recursive construction, by contrast, shows the conjecture in full, but does give analogous bound at intermediate levels $r \leq \alpha(G) - 2$.
\end{remark}

\subsection{Minimizers of the Motzkin-Straus Formulation}
\label{sec:minimizers}
The copositive formulation of $\alpha(G)$ in~\eqref{alpha-COP} can also be viewed as a reformulation of the following quadratic program of Motzkin and Straus~\cite{motzkin-straus}:

\begin{align}\label{motzkin-straus}
    \frac{1}{\alpha(G)}=\min\Big\{x^{\mathsf T}(A_G+I)x \; \ : \  x\in \mathbb{R}_+^n, \ \sum_{i=1}^n x_i=1\Big\}.
\end{align}

Laurent and Vargas~\cite{LV21a} characterize the minimizers of~\eqref{motzkin-straus}. These minimizers are precisely the zeros of $p_G=x^{\mathsf T}M_Gx$ in the standard simplex $\Delta_n=\{x\in \mathbb{R}_+^n: \sum_{i=1}^n x_i=1\}$. We can recover their characterization from Theorem~\ref{thm:main-identity}.

As mentioned in the introduction, the zeros of $p_G$ on $\Delta_n$ form an obstruction to obtaining a certificate of membership of $M_G$ in $\mathcal Q_n^{(r)}$: if $z\in\Delta_n$ satisfies $p_G(z)=0$, then the right-hand side of the identity in Theorem~\ref{thm:main-identity} must also vanish at $z$. Since that right-hand side is a sum of manifestly nonnegative terms (weighted squares and nonnegative-coefficient polynomials), its vanishing forces \emph{every} individual term to vanish, a highly restrictive condition. The structure of the zero set of $p_G$ is therefore encoded in the vanishing terms of the identity, and this suffices to recover the minimizer characterization below.

\begin{theorem}[\cite{LV21a}]
\label{thm:motzkin-straus-zeros}
Let $G=(V,E)$ be a graph, let $x^*\in\Delta_n$, and let
$S=\operatorname{Supp}(x^*)=\{i:x_i^*>0\}$.
Let $C_1,\ldots,C_m$ be the connected components of $G[S]$. Then $x^*$ is a zero of $x^{\mathsf T}M_Gx$ (equivalently, a minimizer of the Motzkin--Straus program~\eqref{motzkin-straus}) if and only if:
\begin{enumerate}
    \item each $C_h$ is a clique of $G$,
    \item $m=\alpha(G)$,
    \item $\displaystyle\sum_{i\in C_h}x_i^*
        =
        \frac{1}{\alpha(G)}
        \qquad\text{for all }h\in[m].$
\end{enumerate}
\end{theorem}

We prove the ``only if'' direction directly from Theorem~\ref{thm:main-identity}. The converse is a standard verification: a clique partition with equal weight $1/\alpha(G)$ on each clique attains the value $1/\alpha(G)$, and can be found in~\cite{LV21a}.

\begin{proof}
Since $x^*$ is a zero of $x^{\mathsf T}M_Gx$ on $\Delta_n$, we have
\[
\sum_i x_i^*=1.
\]
Hence, evaluating Theorem~\ref{thm:main-identity} at $x^*$, the left-hand side is zero. Every term on the right-hand side is nonnegative, since $\lambda_r>0$, $S(x^*)=1$, $x_T^*\geq0$, $L_T(x^*)^2\geq0$, $x_D^*\geq0$, and $\rho_D(x^*)\geq0$. Therefore, every individual term on the right-hand side must vanish.

We first show that every component of $G[S]$ is a clique. It suffices to show that $G[S]$ contains no induced path on three vertices. Suppose that $u,w,v\in S$ satisfy $\{u,w\},\{w,v\}\in E
$ and $\{u,v\}\notin E.$
Then $D=\{u,v\}$ is a stable set, so $D\in\mathcal S_2$. Moreover, $x_D^*=x_u^*x_v^*>0.$
Since $w\notin D$ is adjacent to both vertices of $D$, we have $\rho_D(x^*)\geq x_w^*>0.$
Thus $x_D^*\rho_D(x^*)>0,$
contradicting the fact that every term in the second sum of Theorem~\ref{thm:main-identity} must vanish. Therefore, $G[S]$ contains no induced path on three vertices, and hence every connected component of $G[S]$ is a clique. This proves (1).

Next, fix a component $C_h$ and a vertex $i\in C_h$. Since $C_h$ is a clique, the vertices of $C_h$ are all adjacent to $i$. Therefore, for the singleton stable set $T=\{i\}\in\mathcal S_1$,
\[
\nu_{\{i\}}(x^*)
=
\sum_{\substack{v\neq i\\ \{i,v\}\notin E(G)\\ v\in S}}x_v^*
=
1-\sum_{j\in C_h}x_j^*.
\]
The vanishing of the term corresponding to $T=\{i\}$ in the first sum of Theorem~\ref{thm:main-identity} gives $x_{\{i\}}^*L_{\{i\}}(x^*)^2=0.$
Since $x_i^*>0$, it follows that $L_{\{i\}}(x^*)=0.$ By the definition of $L_{\{i\}}$, this means $(\alpha(G)-1)-\alpha(G)\nu_{\{i\}}(x^*)=0.$
Hence
\[
\nu_{\{i\}}(x^*)
=
\frac{\alpha(G)-1}{\alpha(G)}.
\]
Combining this with $\nu_{\{i\}}(x^*)
=
1-\sum_{j\in C_h}x_j^*$
gives
\[
\sum_{j\in C_h}x_j^*
=
\frac{1}{\alpha(G)}.
\]
Since $h$ was arbitrary, this proves (3).

Finally, since the components $C_1,\ldots,C_m$ partition $S$ and $x^*\in\Delta_n$,
\[
1
=
\sum_{i\in S}x_i^*
=
\sum_{h=1}^m\sum_{i\in C_h}x_i^*
=
\frac{m}{\alpha(G)}.
\]
Therefore, $m=\alpha(G),$ proving (2).
\end{proof}

\section{Hoffman--Pereira matrices}
\label{sec:hoffman-pereira}

Let $G=(V,E)$ be a simple undirected graph. The \emph{graph distance} $\dist_G(i,j)$ between vertices $i,j\in V$ is the length of the shortest path from $i$ to $j$ in $G$; if no such path exists, we set $\dist_G(i,j)=\infty$. A \emph{triangle} in $G$ is a set of three distinct pairwise-adjacent vertices, and we say that $G$ is \emph{triangle-free} if it contains no triangle.

Let $\mathcal E$ denote the set of symmetric matrices $A\in\mathcal S^n$ with $a_{ii}=1$ for every $i$ and $a_{ij}\in\{-1,0,1\}$ for every $i\neq j$. For $A\in\mathcal E$, the \emph{negative-entry graph} of $A$, denoted $G_-(A)$, is the graph on vertex set $[n]$ with edge set
\[
E(G_-(A)) = \{\,ij : i<j,\ a_{ij}=-1\,\}.
\]

Hoffman and Pereira \cite{HoffmanPereira1973} characterize the matrices $A\in\mathcal E$ that are copositive, in terms of the following combinatorial condition on the pattern of $-1$ entries.

\begin{definition}[Hoffman--Pereira sign condition]
\label{def:HP-condition}
A matrix $A\in\mathcal E$ satisfies the \emph{Hoffman--Pereira sign condition} if:
\begin{enumerate}[label=\textup{(HP\arabic*)}]
\item $G_-(A)$ is triangle-free;
\item if $i\neq j$ and $\dist_{G_-(A)}(i,j)=2$, then $a_{ij}=1$.
\end{enumerate}
\end{definition}

\begin{theorem}[Hoffman--Pereira, {\cite[Theorem 3.2]{HoffmanPereira1973}}]
\label{thm:HP-original}
Let $A\in\mathcal E$. Then $A$ is copositive if and only if $A$ satisfies the Hoffman--Pereira sign condition.
\end{theorem}

We first observe that the necessity of conditions (HP1) and (HP2) is easy to prove directly, using the same test vector in both cases. Suppose $A\in\mathcal E$ is copositive. If $G_-(A)$ contains a triangle $\{i,j,k\}$, then setting $x=e_i+e_j+e_k$, we have
\[
x^{\mathsf T}Ax = 3 + 2(a_{ij}+a_{ik}+a_{jk}) = 3-6 = -3 < 0,
\]
contradicting copositivity. This proves (HP1).

If $i\neq j$ satisfy $\dist_{G_-(A)}(i,j)=2$, say via a common neighbor $k$ in $G_-(A)$ (so $a_{ik}=a_{kj}=-1$), then by (HP1) we must have $a_{ij}\neq-1$; suppose toward a contradiction that $a_{ij}=0$. Setting $x=e_i+e_j+e_k$ once more,
\[
x^{\mathsf T}Ax = 3+2(a_{ij}+a_{ik}+a_{jk}) = 3+2(0-1-1) = -1 < 0,
\]
again contradicting copositivity. Hence $a_{ij}=1$, proving (HP2).

The sufficiency direction is the substantial part of the Hoffman--Pereira theorem. In this section we show that every matrix satisfying the Hoffman--Pereira sign condition belongs to some cone $\widetilde{\mathcal Q}_n^{(r)}$, and is thus copositive (Theorem~\ref{thm:sos-certificate} below). This may be regarded as an alternative proof of the Hoffman--Pereira theorem, obtained by an explicit sum-of-squares construction.

\subsection{Recursive identity}
\label{subsec:recursive-identity}

Throughout this section, $A = (a_{ij})_{i,j\in[n]}\in\mathcal E$. For $U\subseteq[n]$ we write $x_U=(x_i)_{i\in U}$ and $A[U]$ for the principal submatrix of $A$ indexed by $U$. We also define
\[
p_U(x) = \sum_{i,j\in U} a_{ij}\,x_i x_j = x_U^{\mathsf T} A[U]\, x_U,
\qquad
S_U(x) = \sum_{i\in U} x_i,
\]
with the convention $p_\emptyset = 0$.

For $i\in U$, define the linear form
\[
L_{U,i}(x) = \sum_{j\in U} a_{ij}\, x_j,
\]
and the \emph{zero-neighborhood} of $i$ inside $U$,
\[
Z_U(i) = \{\, j\in U\setminus\{i\} : a_{ij}=0 \,\}.
\]

For a three-element subset $\{i,j,k\}\subseteq U$, define
\begin{equation}
\label{eq:gamma-def}
\gamma_{ijk} \;=\;
a_{ij}+a_{ik}+a_{jk}
-a_{ij}a_{ik}-a_{ij}a_{jk}-a_{ik}a_{jk}
-\one_{\{a_{ij}=a_{ik}=0\}}a_{jk}
-\one_{\{a_{ij}=a_{jk}=0\}}a_{ik}
-\one_{\{a_{ik}=a_{jk}=0\}}a_{ij},
\end{equation}
where $\one_{\{\mathcal P\}}$ equals $1$ if the statement $\mathcal P$ holds and $0$ otherwise. Since $A\in\mathcal E$ forces every pairwise value in \eqref{eq:gamma-def} to lie in $\{-1,0,1\}$, and \eqref{eq:gamma-def} treats the three pairs symmetrically, $\gamma_{ijk}$ depends only on the unordered triple of values $\{a_{ij},a_{ik},a_{jk}\}$, not on any further structure of $U$ or of the index set $[n]$.

Finally, set
\[
R_U(x) = 2\!\!\sum_{\{i,j,k\}\subseteq U}\!\! \gamma_{ijk}\, x_i x_j x_k.
\]

We first prove that, under the Hoffman--Pereira sign condition, $R_U$ has nonnegative coefficients for all $U\subseteq[n]$.

\begin{lemma}
\label{lem:gamma-nonnegative}
Assume that $A$ satisfies the Hoffman--Pereira sign condition. Then every coefficient $\gamma_{ijk}$ of $R_U$ is nonnegative, for every $U\subseteq[n]$ and every triple $\{i,j,k\}\subseteq U$.
\end{lemma}

\begin{proof}
Fix three distinct vertices $i,j,k\in U$. Since $\gamma_{ijk}$ depends only on the unordered multiset of values $\{a_{ij},a_{ik},a_{jk}\}$, it suffices to inspect this multiset. Under the Hoffman--Pereira sign condition, the possible multisets and the corresponding values of $\gamma_{ijk}$ are:
\[
\begin{array}{c|c}
\text{multiset } \{a_{ij},a_{ik},a_{jk}\} & \gamma_{ijk} \\
\hline
\{1,1,1\} & 0 \\
\{1,1,0\} & 1 \\
\{1,0,0\} & 0 \\
\{0,0,0\} & 0 \\
\{-1,1,1\} & 2 \\
\{-1,1,0\} & 1 \\
\{-1,0,0\} & 0 \\
\{-1,-1,1\} & 0
\end{array}
\]
No other multiset can occur. Indeed, three $-1$ entries would form a triangle in $G_-(A)$, which is excluded by the sign condition. If exactly two of the three entries equal $-1$, the endpoints of the corresponding two negative edges lie at distance two in $G_-(A)$, so the sign condition forces the third entry to equal $+1$. This rules out, e.g., $\{-1,-1,0\}$. Since every value of $\gamma_{ijk}$ occurring above is nonnegative, the lemma follows.
\end{proof}

We now prove the main recursive identity underlying our certificate.

\begin{theorem}
\label{thm:local-identity}
For every $U\subseteq[n]$,
\begin{equation}
\label{eq:local-identity}
S_U(x)\,p_U(x)
= \sum_{i\in U} x_i\, L_{U,i}(x)^2
+ \sum_{i\in U} x_i\, p_{Z_U(i)}(x)
+ R_U(x).
\end{equation}
\end{theorem}

\begin{proof}
Both sides of \eqref{eq:local-identity} are homogeneous cubic polynomials in the variables $(x_i)_{i\in U}$. We verify the identity by comparing coefficients on both sides.

\smallskip
\noindent\textit{Coefficient of $x_i^3$.}
On the left-hand side the coefficient is $a_{ii}=1$. On the right, only the term $x_i L_{U,i}(x)^2$ can produce $x_i^3$ (through the $j=\ell=i$ summand of $L_{U,i}(x)^2$), contributing $a_{ii}^2=1$. Neither $\sum_i x_i\,p_{Z_U(i)}(x)$ nor $R_U(x)$ contains an $x_i^3$ term: the former because $i\notin Z_U(i)$, and the latter because $R_U$ is by definition supported on monomials $x_ix_jx_k$ with $i,j,k$ pairwise distinct.

\smallskip
\noindent\textit{Coefficient of $x_i^2x_j$, $i\neq j$.}
Write $a := a_{ij}\in\{-1,0,1\}$. The left-hand side contributes $1+2a$. On the right, $\sum_\ell x_\ell L_{U,\ell}(x)^2$ contributes $2a$ from $\ell=i$ and $a^2$ from $\ell = j$. The term $\sum_\ell x_\ell\, p_{Z_U(\ell)}(x)$ contributes $\one_{\{a=0\}}$ (only $\ell=j$ can supply this monomial, and only when $i\in Z_U(j)$, i.e.\ $a=0$). By construction $R_U$ has no $x_i^2x_j$ monomial, since its sum ranges only over triples of \emph{distinct} indices. Hence \eqref{eq:local-identity}, restricted to this monomial, reduces to the numerical identity
\[
1+2a = \bigl(2a+a^2\bigr) + \one_{\{a=0\}},
\qquad\text{equivalently}\qquad
1-a^2-\one_{\{a=0\}}=0.
\]
This holds for every $a\in\{-1,0,1\}$, since on this range $\one_{\{a=0\}}=1-a^2$.

\smallskip
\noindent\textit{Coefficient of $x_ix_jx_k$, $i,j,k$ distinct.}
The left-hand side contributes
\[
2\bigl(a_{ij}+a_{ik}+a_{jk}\bigr).
\]
On the right, $\sum_\ell x_\ell L_{U,\ell}(x)^2$ contributes
\[
2\bigl(a_{ij}a_{ik}+a_{ij}a_{jk}+a_{ik}a_{jk}\bigr),
\]
and $\sum_\ell x_\ell\, p_{Z_U(\ell)}(x)$ contributes
\[
2\,\one_{\{a_{ij}=a_{ik}=0\}}a_{jk}
+2\,\one_{\{a_{ij}=a_{jk}=0\}}a_{ik}
+2\,\one_{\{a_{ik}=a_{jk}=0\}}a_{ij}.
\]
The coefficient required from $R_U$ to balance the identity is therefore precisely $2\gamma_{ijk}$, matching \eqref{eq:gamma-def} by definition.
\end{proof}

\subsection{Sum-of-squares certificate for Hoffman--Pereira matrices}
\label{subsec:tildeQ-membership-simple}

We now show that every matrix satisfying the Hoffman--Pereira sign condition lies in $\widetilde{\mathcal Q}_n^{(r)}$ for some $r\in\mathbb N$, and thus admits a sum-of-squares certificate of copositivity. Recall that $\mathcal N_{n,r}$ denotes the cone of homogeneous polynomials of degree $r$ in $n$ variables with nonnegative coefficients. Consider the set
\[
\mathcal H_{n,r} = \Big\{\, \sum_{\substack{\beta \in \mathbb{N}^n \\ |\beta| \in \{r,\,r+2\}}} x^\beta\, \sigma_\beta \;:\; \sigma_\beta \in \Sigma_{n,\, r+2-|\beta|} \,\Big\}.
\]
The cones $\widetilde{\mathcal Q}_n^{(r)}$ (see~\eqref{eq:tildeQ}) are then given by
\begin{equation}
\widetilde{\mathcal Q}_n^{(r)} = \Big\{\, M \in \mathcal S^n \;:\; p(x)\, x^{\mathsf T} M x \in \mathcal H_{n,r},\ p\in\mathcal N_{n,r},\ \|p\|_1 = 1 \,\Big\}.
\end{equation}

The following absorption property follows directly from the definitions of $\mathcal N_{n,r}$ and $\mathcal H_{n,r}$.

\begin{lemma}
\label{lem:absorption}
If $q\in\mathcal H_{n,r}$ and $m\in\mathcal N_{n,s}$, then $mq\in\mathcal H_{n,r+s}$. Moreover, $\mathcal H_{n,r}$ is closed under addition.
\end{lemma}

\begin{proof}
Write $q=\sum_{|\beta|=r} x^\beta\sigma_\beta + \sum_{|\beta|=r+2} c_\beta x^\beta$ with $\sigma_\beta\in\Sigma_{n,2}$, $c_\beta\ge0$, as in the definition of $\mathcal H_{n,r}$. Expanding $m=\sum_{|\gamma|=s} d_\gamma x^\gamma$ with $d_\gamma\ge0$, we have
\[
mq = \sum_{|\beta|=r}\ \sum_{|\gamma|=s} d_\gamma\, x^{\beta+\gamma}\,\sigma_\beta
\;+\; \sum_{|\beta|=r+2}\ \sum_{|\gamma|=s} d_\gamma c_\beta\, x^{\beta+\gamma}.
\]
Each monomial $x^{\beta+\gamma}$ in the first sum has degree $r+s$, and its coefficient $d_\gamma\sigma_\beta$ remains in $\Sigma_{n,2}$ since $d_\gamma\ge0$. Each monomial in the second sum has degree $r+s+2$ with nonnegative coefficient $d_\gamma c_\beta$. Hence $mq\in\mathcal H_{n,r+s}$. Closure under addition is immediate.
\end{proof}

We now prove that every principal submatrix $A[U]$ of a Hoffman--Pereira matrix admits a certificate of the type $A[U]\in \widetilde{\MQ}^{(r)}$. Taking $U=[n]$ then yields our main theorem.

\begin{proposition}
\label{prop:existence}
Let $A\in\mathcal E$ satisfy the Hoffman--Pereira sign condition. Then, for every $U\subseteq[n]$, there exist $r_U\ge0$ and a nonzero $M_U\in\mathcal N_{|U|,r_U}$ such that
\[
M_U(x)\,p_U(x) \;\in\; \mathcal H_{|U|,r_U}.
\]
\end{proposition}

\begin{proof}
We argue by strong induction on $|U|$. If $|U|\le1$, take $M_U=1\in\mathcal N_{|U|,0}$. Then $M_Up_U=p_U$, and for $|U|=0$ this is the zero polynomial (trivially in $\mathcal H_{0,0}$), while for $|U|=1$, $p_U(x)=x_i^2\in\mathcal H_{1,0}$.

Assume $|U|\ge2$ and that the claim holds for every proper subset of $U$. In particular, since $Z_U(i)\subsetneq U$ for all $i\in U$, for each $i\in U$ there exist $r_i\ge0$ and a nonzero $M_{Z_U(i)}\in\mathcal N_{|Z_U(i)|,r_i}$ with $M_{Z_U(i)}\,p_{Z_U(i)}\in\mathcal H_{|Z_U(i)|,r_i}$. By Theorem~\ref{thm:local-identity},
\[
S_U(x)\,p_U(x) = \underbrace{\sum_{i\in U} x_iL_{U,i}(x)^2 + R_U(x)}_{=:B_U(x)} \;+\; \sum_{i\in U} x_i\,p_{Z_U(i)}(x).
\]
Since each $x_iL_{U,i}(x)^2$ is a degree-$1$ monomial times a square in $\Sigma_{|U|,2}$, and $R_U(x)$ is a degree-$3$ polynomial with nonnegative coefficients by Lemma~\ref{lem:gamma-nonnegative}, we have $B_U\in\mathcal H_{|U|,1}$.

Let $P:=\prod_{i\in U}M_{Z_U(i)}(x)\in\mathcal N_{|U|,\,R}$, where $R:=\sum_{i\in U}r_i$. Multiplying the identity above by $P$ gives
\[
S_U(x)\,P(x)\,p_U(x) = P(x)\,B_U(x) \;+\; \sum_{i\in U} x_i\,P(x)\,p_{Z_U(i)}(x).
\]
By Lemma~\ref{lem:absorption} (with $m=P\in\mathcal N_{|U|,R}$, $q=B_U\in\mathcal H_{|U|,1}$), the first term satisfies $P\,B_U\in\mathcal H_{|U|,R+1}$.

For each $i\in U$, write $P = M_{Z_U(i)}\cdot\big(\prod_{k\ne i}M_{Z_U(k)}\big)$, so that
\[
x_i\,P(x)\,p_{Z_U(i)}(x) = \Big(x_i\prod_{k\ne i}M_{Z_U(k)}(x)\Big)\cdot\Big(M_{Z_U(i)}(x)\,p_{Z_U(i)}(x)\Big).
\]
The first factor is a nonzero element of $\mathcal N_{|U|,\,R-r_i+1}$, and the second is, by the induction hypothesis, an element of $\mathcal H_{|Z_U(i)|,r_i}\subseteq\mathcal H_{|U|,r_i}$. By Lemma~\ref{lem:absorption} again, the product lies in $\mathcal H_{|U|,R+1}$.

Summing over $i\in U$ and adding the $P\,B_U$ term (using closure of $\mathcal H_{|U|,R+1}$ under addition, Lemma~\ref{lem:absorption}), we conclude
\[
S_U(x)\,P(x)\,p_U(x) \;\in\; \mathcal H_{|U|,R+1}.
\]
Setting $M_U := S_U\cdot P$, a nonzero element of $\mathcal N_{|U|,R+1}$ (nonzero since $S_U$ and $P$ are nonzero, nonnegative-coefficient homogeneous polynomials), this is exactly $M_U\,p_U\in\mathcal H_{|U|,R+1}$, completing the induction with $r_U=R+1$.
\end{proof}

\begin{theorem}
\label{thm:sos-certificate}
Let $A\in\mathcal E$ satisfy the Hoffman--Pereira sign condition. Then there exists $r\ge0$ such that
\[
A \in \widetilde{\mathcal Q}_n^{(r)}.
\]
\end{theorem}

\begin{proof}
Apply Proposition~\ref{prop:existence} with $U=[n]$: there exist $r\ge0$ and a nonzero $M\in\mathcal N_{n,r}$ such that $M(x)\,\big(x^{\mathsf T}Ax\big)\in\mathcal H_{n,r}$. Set $p:=M/\|M\|_1\in\mathcal N_{n,r}$, so that $\|p\|_1=1$. Dividing through by $\|M\|_1>0$, and using that $\mathcal H_{n,r}$ is closed under multiplication by positive scalars, gives
\[
p(x)\,\big(x^{\mathsf T}Ax\big) \;\in\; \mathcal H_{n,r},
\]
which shows that $A\in\widetilde{\mathcal Q}_n^{(r)}$.
\end{proof}

\subsection{Membership in the cones \texorpdfstring{$\MK_n^{(r)}$ and $\MQ_n^{(r)}$}{K and Q}}

In general, matrices satisfying the Hoffman--Pereira sign condition do not belong to the cones $\MK_n^{(r)}$.
\begin{remark}
\label{ex:horn-plus-edge}
Let $H\in\rev{\mathcal S^5}{\mathbb S^5}$ be the Horn matrix,
\[
H = \begin{pmatrix}
1 & 1 & -1 & -1 & 1 \\
1 & 1 & 1 & -1 & -1 \\
-1 & 1 & 1 & 1 & -1 \\
-1 & -1 & 1 & 1 & 1 \\
1 & -1 & -1 & 1 & 1
\end{pmatrix},
\]
and let $M_2=\begin{pmatrix}1&-1\\-1&1\end{pmatrix}\in\rev{\mathcal S^2}{\mathbb S^2}$. Consider the direct sum
\[
H\oplus M_2 = \begin{pmatrix} H & 0 \\ 0 & M_2 \end{pmatrix} \in \mathcal S^{7}.
\]

This matrix satisfies the Hoffman--Pereira sign condition. Its diagonal is all $1$'s and every off-diagonal entry lies in $\{-1,0,1\}$. The negative-entry graph $G_-(H\oplus M_2)$ is the disjoint union of the $5$-cycle $1\text{-}4\text{-}2\text{-}5\text{-}3\text{-}1$ (from $H$) and the single edge $\{6,7\}$ (from $M_2$), which is triangle-free (HP1). It is straightforward to verify that the vertices at distance two in this $5$-cycle are exactly the pairs $\{i,j\}\subseteq\{1,\dots,5\}$ with $H_{ij}=1$, so (HP2) holds within the $C_5$ block. The two-vertex component $\{6,7\}$ has no pair at distance two, so (HP2) holds within the $M_2$ block. Vertices in distinct components are at infinite distance, so (HP2) imposes no condition across the two blocks. By Theorem~\ref{thm:sos-certificate}, it follows that
\[
H\oplus M_2 \in \widetilde{\mathcal Q}_7^{(r)}
\]
for some $r\ge0$. In fact, $r=1$ suffices with the multiplier $\sum_{i=1}^5x_i$: the Horn matrix is the graph matrix of a $5$-cycle and hence lies in $\mathcal Q_5^{(1)}$ by Theorem~\ref{thm:main-identity}, while multiplying the positive semidefinite form $(x_6-x_7)^2$ by this nonnegative-coefficient linear form preserves the required decomposition.

On the other hand, by \cite[Theorem~3]{lv21b}, we have
\[
H\oplus M_2 \;\notin\; \bigcup_{r\ge0}\mathcal K_7^{(r)},
\]
and hence also $H\oplus M_2\notin\bigcup_{r\ge0}\mathcal Q_7^{(r)}$, since $\mathcal Q_7^{(r)}\subseteq\mathcal K_7^{(r)}$ for every $r$.
\end{remark}

\begin{remark}
Dickinson \emph{et al.}~\cite{DDGH} proved that every $5\times5$ copositive matrix with unit diagonal belongs to $\MK_5^{(1)}$. In particular, every $5\times5$ matrix satisfying the Hoffman--Pereira sign condition belongs to $\MK_5^{(1)}$. Lemma~\ref{lem:four-zeros} below strengthens this conclusion for Hoffman--Pereira matrices by placing them in $\MQ_5^{(1)}$. Remark~\ref{ex:horn-plus-edge} gives a $7\times7$ matrix satisfying the Hoffman--Pereira condition that does not belong to any cone $\MK^{(r)}$ (and hence to any $\MQ^{(r)}$). For every $n>7$, taking its direct sum with $I_{n-7}$ gives an $n\times n$ example: the resulting matrix still satisfies the sign condition, while membership in any $\MK_n^{(r)}$ would imply membership of its $7\times7$ principal submatrix in $\MK_7^{(r)}$. This argument does not resolve the case $n=6$. 
\end{remark}

We finish by giving a sufficient condition for Hoffman--Pereira matrices to belong to $\MQ_n^{(1)}$.

\begin{lemma}
\label{lem:four-zeros}
Let $A\in\mathcal E$ satisfy the Hoffman--Pereira sign condition, and for $i\in[n]$ let
\[
Z(i) = \{\, j\in[n]\setminus\{i\} : a_{ij}=0 \,\}
\]
denote the zero-neighborhood of $i$ in $A$. If $|Z(i)| \le 4$ for every  $i\in[n]$, then $A \in \mathcal{Q}_n^{(1)}$.

\end{lemma}

\begin{proof}
Apply Theorem~\ref{thm:local-identity} with $U=[n]$:
\begin{equation}
\label{eq:full-identity}
\Big(\sum_{i=1}^n x_i\Big)\, x^{\mathsf T}Ax
= \sum_{i=1}^n x_i\, L_i(x)^2
+ \sum_{i=1}^n x_i\, p_{Z(i)}(x)
+ R(x),
\end{equation}
where  $R(x)$ has nonnegative coefficients by Lemma~\ref{lem:gamma-nonnegative}.

Fix $i\in[n]$. Since $A[Z(i)]$ is a copositive matrix and $|Z(i)|\le4$, Diananda's theorem~\cite{Dia} gives $A[Z(i)]=P_i+N_i$, where $P_i$ is positive semidefinite and $N_i$ is entrywise nonnegative. Thus $p_{Z(i)}(x)=x_{Z(i)}^{\mathsf T}A[Z(i)]x_{Z(i)}=x_{Z(i)}^{\mathsf T}P_ix_{Z(i)}+x_{Z(i)}^{\mathsf T}N_ix_{Z(i)}$.
Hence
\[
x_i\,p_{Z(i)}(x) = x_i\big(x_{Z(i)}^{\mathsf T}P_ix_{Z(i)}\big) + x_i\big(x_{Z(i)}^{\mathsf T}N_ix_{Z(i)}\big)
\]
The terms $x_iL_i(x)^2$ are degree-$1$ monomials times squares of linear forms, and $R(x)$ is a cubic with nonnegative coefficients. Therefore, the right-hand side of~\eqref{eq:full-identity} has the form
\[
\sum_{\substack{\beta \in \mathbb{N}^n \\ |\beta| \in \{1,3\}}} x^\beta \sigma_\beta,
\qquad \sigma_\beta \in \Sigma_{n,\,3-|\beta|},
\]
which is precisely a certificate that $A\in \mathcal{Q}_n^{(1)}$.
\end{proof}
\begin{example}[The standard $7\times7$ Hoffman--Pereira matrix]
\label{ex:hp7}
Let $A\in\mathcal E$ be the circulant matrix with first row $(1,-1,1,0,0,1,-1)$,
\[
A = \begin{pmatrix}
1&-1&1&0&0&1&-1\\
-1&1&-1&1&0&0&1\\
1&-1&1&-1&1&0&0\\
0&1&-1&1&-1&1&0\\
0&0&1&-1&1&-1&1\\
1&0&0&1&-1&1&-1\\
-1&1&0&0&1&-1&1
\end{pmatrix},
\]
the extremal matrix appearing in \cite[Section 5]{HoffmanPereira1973}. Its negative-entry graph $G_-(A)$ is the $7$-cycle $1\text{-}2\text{-}\cdots\text{-}7\text{-}1$: it is triangle-free (HP1), and every pair at distance $2$ (i.e.\ $i,i\pm2$ modulo $7$) has entry $+1$ (HP2).

For $U=[7]$, each zero-neighborhood is a single negative edge, $Z_{[7]}(i)=\{i+3,i+4\}$ (mod $7$), so $p_{Z_{[7]}(i)}(x)=(x_{i+3}-x_{i+4})^2$ is already a perfect square. Hence the identity of Theorem \ref{thm:local-identity} with $U=[n]$ is already of the required form:
\[
\big(\sum_{i=1}^7 x_i\big)\,\big(x^{\mathsf T}Ax\big)
= \sum_{i=1}^{7} x_i\,L_i(x)^2
+ \sum_{i=1}^{7} x_i\,(x_{i+3}-x_{i+4})^2
+ 2\!\!\sum_{\{i,j,k\}\in\mathcal T}\!\! x_ix_jx_k,
\]
where indices are interpreted modulo $7$ with representatives in $\{1,\dots,7\}$, $L_i(x)=x_i-x_{i-1}-x_{i+1}+x_{i-2}+x_{i+2}$, and $\mathcal T=\{\{i,j,k\}:\gamma_{ijk}=1\}$ consists of $21$ triples: $14$ of type $\{-1,0,1\}$ and $7$ of type $\{0,1,1\}$ in the notation of Lemma~\ref{lem:gamma-nonnegative}.
\end{example}

\section*{Acknowledgements}
The second author was supported by the AI Interdisciplinary Institute ANITI 
funding through the French ``France 2030" program under the Grant 
agreement No. ANR-23-IACL-0002.

\bibliographystyle{abbrv}  
\bibliography{references}

\end{document}